\documentclass[oneside,english]{amsart}
\usepackage[T1]{fontenc}
\usepackage[latin9]{inputenc}
\usepackage{amsthm}
\usepackage{amstext}
\usepackage{amssymb}
\usepackage{esint}
\usepackage{graphicx}
\usepackage{enumerate}

\usepackage{xcolor}
\usepackage{graphicx}
\graphicspath{{noiseimages/}}

\makeatletter
\newtheorem{theorem}{Theorem}[section]
\newtheorem{lemma}[theorem]{Lemma}
\newtheorem{proposition}[theorem]{Proposition}

\numberwithin{figure}{section}
\theoremstyle{definition}

\newtheorem{example}[theorem]{Example}

\theoremstyle{remark}
\newtheorem{remark}[theorem]{Remark}

\numberwithin{equation}{section}
\makeatother

\usepackage{babel}

	\DeclareMathOperator{\dist}{dist}
	\DeclareMathOperator{\loc}{loc}
	
	\DeclareMathOperator*{\esssup}{ess\,sup}

\begin{document}

\title[Composition operators with applications to Neumann eigenvalues]{Space quasiconformal composition operators with applications to Neumann eigenvalues}

\author{V.~Gol'dshtein, R.~Hurri-Syrj\"anen, V.~Pchelintsev, A.~Ukhlov}

\begin{abstract}
In this article we obtain estimates of Neumann eigenvalues of $p$-Laplace operators in a large class of space domains satisfying  quasihyperbolic boundary conditions. The suggested method is based on composition operators generated by quasiconformal mappings and their applications to Sobolev-Poincar\'e-inequalities.  By using a sharp version of the inverse H\"older inequality we refine our estimates for quasi-balls, that is, images of balls under quasiconformal mappings of the whole space.
\end{abstract}

\maketitle

\section{Introduction}

The article is devoted to  applications of the space quasiconformal mappings theory to the spectral theory of elliptic operators. Applications are based on the geometric theory of composition operators on Sobolev spaces \cite{GU10,U93,VU02} in the special case of operators generated by quasiconformal mappings.
Composition operators on Sobolev spaces permit us to give estimates of norms for embedding operators of Sobolev spaces into Lebesgue spaces in a  large class of space domains that includes domains with H\"older singularities \cite{GU17,GU19}. Quasiconformal mappings allow us to describe the important subclass of these embedding domains in the terms of quasihyperbolic geometry. It permits us to obtain estimates of Neumann eigenvalues of the $p$-Laplace operator, $p> n$, in domains with quasihyperbolic boundary conditions. Note, that domains with quasihyperbolic boundary conditions are important subclass of Gromov hyperbolic domains \cite{BHK}. 

These estimates are refined for $K$-quasi-balls, that is, images of the ball $\mathbb B \subset\mathbb R^n$ under $K$-quasiconformal mappings $\varphi:\mathbb R^n\to\mathbb R^n$ with the help of a sharp (with constant estimates) reverse H\"older inequality. We prove the constant estimates in the  reverse H\"older inequality on the base of the quasiconformal mapping theory \cite{BI} and the non-linear potential theory \cite{HKM,M}.


Namely we prove that if a domain $\Omega\subset\mathbb R^n$, $n\geq 3$, satisfying the $\gamma$-quasihyperbolic boundary condition, is  a $K$-quasi-ball, then for $p>n$
$$
\mu_p(\Omega)\geq \frac{M_p(K,\gamma)}{R_{*}^p},
$$
where $R_{\ast}$ is a radius of a ball $\Omega^{\ast}$ of the same measure as $\Omega$ and $M_p(K,\gamma)$ depends only on $p,\gamma$ and a quasiconformity coefficient $K$ of $\Omega$. The exact value of $M_p(K,\gamma)$ is given in Theorem~\ref{main}.

Estimates of the first non-trivial Neumann eigenvalue of the $p$-Laplace operator, $p>2$, are known for convex domains
$\Omega\subset\mathbb R^n$  \cite{ENT}:
$$
\mu_p(\Omega) \geq \left(\frac{\pi_p}{d(\Omega)}\right)^p ,
$$
where $d(\Omega)$ is a diameter of a convex domain $\Omega$ and $\pi_p=2 \pi {(p-1)^{\frac{1}{p}}}/({p \sin(\pi/p))}$.

Unfortunately in  non-convex domains $\mu_p(\Omega)$ can not be characterized in the terms of Euclidean diameters. This can be seen by considering a domain consisting of two identical squares connected by a thin corridor \cite{BCDL16}. The method of the composition operators on Sobolev spaces in connection with embedding theorems in convex domains \cite{GG94,GU09} allows us to obtain estimates of Neumann eigenvalues of the $p$-Laplace operator in non-convex domains including some domains with H\"older singularities and some fractal type domains \cite{GU16,GU17}.

In the case of composition operators generated by quasiconformal mappings the main estimates of norms of embedding operators from Sobolev spaces with first derivatives into Lebesgue spaces were reformulated in the terms of integrals of quasiconformal derivatives \cite{GH-SU18,GPU18_1,GPU18}. This type of global integrability of quasiconformal derivatives depends on the quasiconformal geometry of domains \cite{G73}. One of the possible geometric reinterpretation of the quasiconformal geometry is the growth condition for quasihyperbolic metric \cite{AK, HS}.

Recall that a domain $\Omega$ satisfies the $\gamma$-quasihyperbolic boundary condition with some $\gamma>0$, if the growth condition on the quasihyperbolic metric
$$
k_{\Omega}(x_0,x)\leq \frac{1}{\gamma}\log\frac{\dist(x_0,\partial\Omega)}{\dist(x,\partial\Omega)}+C_0
$$
is satisfied for all $x\in\Omega$, where $x_0\in\Omega$ is a fixed base point and $C_0=C_0(x_0)<\infty$,
\cite{GM, H1, HSMV, KOT}.

In \cite{AK} it was proved that Jacobians $J_{\varphi}$ of  quasiconformal mappings $\varphi: \mathbb{B} \to \Omega$ belong to $L_{\beta}(\mathbb{B})$ for some $\beta>1$ if and only if $\Omega$ satisfy to a $\gamma$-quasihyperbolic boundary conditions for some $\gamma$.

Hence the (quasi)conformal mapping theory allows to obtain spectral estimates in domains with quasihyperbolic boundary conditions. Because we need the exact value of the integrability exponent $\beta$ for quasiconformal Jacobians, we consider an equivalent class of $\beta$-quasiconformal regular domains \cite{GPU19}, namely the class of domains
$$
\left\{\Omega\subset\mathbb R^n : \Omega=\varphi(\mathbb B)\,\,\,\text{with}\,\,\, J_{\varphi}\in L_{\beta}(\mathbb{B})\right\}.
$$

An important subclass of $\beta$-quasiconformal regular domains are quasi-balls, because
Jacobians of quasiconformal mappings are $A_p$-weights \cite{S93} and they satisfy the reverse H\"older inequality \cite{BI}. Methods of the harmonic analysis allow us to refine estimates of Neumann eigenvalues in the case of quasi-balls. The main technical difficulties are calculation (estimation) of exact constants in the reverse H\"older inequality. We solve this problem using capacitary (moduli) estimates. These estimates permit us to obtain integral estimates of quasiconformal derivatives (Jacobians) in the unit ball $\mathbb B$. Hence we can refine estimates of the Neumann eigenvalues for quasi-balls. Note that quasi-balls include some fractal type domains.

In the two-dimensional case $\mathbb R^2$ this approach is more accurate \cite{GPU18_2,GPU18} because exact exponents of local integrability of planar quasiconformal Jacobians are known \cite{A94,G81}.

\section{Composition operators on Sobolev spaces}

\subsection{Sobolev spaces}

Let $\Omega$ is an open subset of $\mathbb R^n$, $n\geq 2$. The Sobolev space $W^1_p(\Omega)$, $1\leq p\leq\infty$ is defined
as a Banach space of locally integrable weakly differentiable functions
$f:\Omega\to\mathbb{R}$ equipped with the following norm:
\[
\|f\mid W^1_p(\Omega)\|=\| f\mid L_p(\Omega)\|+\|\nabla f\mid L_p(\Omega)\|,
\]
where $\nabla f$ is the weak gradient of the function $f$. The homogeneous seminormed Sobolev space $L^1_p(\Omega)$, $1\leq p\leq\infty$ is  equipped with the seminorm:
\[
\|f\mid L^1_p(\Omega)\|=\|\nabla f\mid L_p(\Omega)\|.
\]

We consider the Sobolev spaces as Banach spaces of equivalence classes of functions up to a set of $p$-capacity zero \cite{M}.

Recall that a homeomorphism $\varphi: \Omega\to \Omega'$ is called a $K$-quasiconformal mapping if $\varphi\in W^1_{n,\loc}(\Omega)$ and there exists a constant $1\leq K<\infty$ such that
$$
|D\varphi(x)|^n\leq K |J(x,\varphi)|\,\,\text{for almost all}\,\,x\in\Omega.
$$
The quasiconformal mappings are mappings of finite distortion and possess the Luzin $N$-property, that is, an image of a set of measure zero has measure zero.

Let $\Omega$ and $\Omega'$ be domains in $\mathbb R^n$, that is, open and connected sets.  Then $\Omega$ and $\Omega'$ are called $K$-quasiconformal equivalent domains if there exists a $K$-quasiconformal homeomorphism $\varphi: \Omega\to \Omega'$. Note that in $\mathbb R^2$ any two simply connected domains are quasiconformal equivalent domains \cite{Ahl66}.

In the following theorems we obtain the estimate of the norm of the composition operator on Sobolev spaces in quasiconformal equivalent domains with finite measure. The case of planar conformal mappings was considered in \cite{GPU18}.

\begin{theorem}
\label{NormCO}
Let $\Omega, \Omega'\subset \mathbb R^n$, $n\geq 2$, be $K$-quasiconformal equivalent domains with finite measure. Then a $K$-quasiconformal mapping $\varphi: \Omega\to \Omega'$ generates a bounded composition operator
\[
\varphi^{*}:L^1_p(\Omega') \to L^1_q(\Omega)
\]
for any $p \in (n,\, +\infty)$ and $q\in[1,\, n]$ with $K_{p,q}(\Omega)=K^{\frac{1}{n}} |\Omega'|^{\frac{p-n}{np}} |\Omega|^{\frac{n-q}{nq}}$.
\end{theorem}

\begin{proof}
By using the quasiconformal inequality $|D\varphi(x)|^n\leq K |J(x,\varphi)|$ for almost all $x\in\Omega$ we obtain
\begin{equation*}
K_{p,q}(\Omega)=\left(\int\limits_\Omega \left(\frac{|D\varphi(x)|^p}{|J(x,\varphi)|}\right)^\frac{q}{p-q}~dx\right)^\frac{p-q}{pq} \\
\leq K^{\frac{1}{p}}
\left(\int\limits_\Omega |D\varphi(x)|^{\frac{(p-n)q}{p-q}}~dx\right)^\frac{p-q}{pq}.
\end{equation*}
Note that if $q \leq n$ then the quality $(p-n)q/(p-q) \leq n$. Hence applying the
H\"older inequality to the last integral we have
\begin{multline*}
\left(\int\limits_\Omega |D\varphi(x)|^{\frac{(p-n)q}{p-q}}~dx\right)^\frac{p-q}{pq}
\leq \left[
\left(\int\limits_\Omega |D\varphi(x)|^n~dx\right)^{\frac{(p-n)q}{n(p-q)}}
\left(\int\limits_\Omega dx\right)^{\frac{(n-q)p}{n(p-q)}}
\right]^\frac{p-q}{pq} \\
\leq K^{\frac{p-n}{np}} \left[
\left(\int\limits_\Omega |J(x,\varphi)|~dx\right)^{\frac{(p-n)q}{n(p-q)}}
\left(\int\limits_\Omega dx\right)^{\frac{(n-q)p}{n(p-q)}}
\right]^\frac{p-q}{pq}.
\end{multline*}
By the condition of the theorem,  $\Omega$ and $\Omega'$ are Euclidean  domains with
finite measure and therefore
\[
K_{p,q}(\Omega)\leq K^{\frac{1}{p}}K^{\frac{p-n}{np}} |\Omega'|^{\frac{p-n}{np}} |\Omega|^{\frac{n-q}{nq}} = K^{\frac{1}{n}} |\Omega'|^{\frac{p-n}{np}} |\Omega|^{\frac{n-q}{nq}}<\infty.
\]

Hence, by \cite{U93,VU02} we obtain that a composition operator
\[
\varphi^{*}:L^1_p(\Omega') \to L^1_q(\Omega)
\]
is bounded for any $p \in (n,\, +\infty)$ and $q\in[1,\, n]$.
\end{proof}

Let $\Omega$ and $\Omega'$ be domains in $\mathbb R^n$, $n\geq 2$. Then a domain $\Omega'$ is called a $K$-quasi\-con\-for\-mal $\beta$-regular domain about a domain $\Omega$ if there exists a $K$-quasiconformal mapping $\varphi : \Omega \to \Omega'$ such that
$$
\int\limits_\Omega |J(x, \varphi)|^{\beta}~dx < \infty \quad\text{for some}\quad \beta >1,
$$
where $J(x,\varphi)$ is a Jacobian of a $K$-quasiconformal mapping $\varphi : \Omega \to \Omega'$.
The domain $\Omega' \subset \mathbb{R}^n$ is called a quasiconformal regular domain if it is a $K$-quasiconformal $\beta$-regular domain for some $\beta>1$.

Note that the class of quasiconformal regular domains includes the class of Gehring domains \cite{AK, HS} and can be described in terms of quasihyperbolic geometry \cite{GM, H1, KOT}.



Now we establish a connection between quasiconformal $\beta$-regular domains
and the quasiconformal composition operators on Sobolev spaces.

\begin{theorem}\label{th-conn}
Let $\Omega, \Omega'\subset \mathbb R^n$, $n\geq 2$, be domains. Then $\Omega'$ is a $K$-quasiconformal $\beta$-regular domain about a domain $\Omega$ if and only if a $K$-quasiconformal mapping $\varphi: \Omega\to \Omega'$ generates a bounded composition operator
\[
\varphi^{*}:L^1_p(\Omega') \to L^1_q(\Omega)
\]
for any $p \in (n, +\infty)$ and $q=p\beta/(p+\beta -n)$.
\end{theorem}

\begin{proof}
Because $\varphi$ is a quasiconformal mapping, then $\varphi\in W^1_{n,\loc}(\Omega)$ and Jacobian $J(x,\varphi)\ne 0$ for almost all $x\in\Omega$. Hence the $p$-dilatation
$$
K_p(x)=\frac{|D\varphi(x)|^p}{|J(x,\varphi)|}
$$
is well defined for almost all $x\in\Omega$ and so $\varphi$ is a mapping of finite distortion.

Let $\Omega'$ be a $K$-quasiconformal $\beta$-regular domain about a domain $\Omega$. Then there exists $K$-quasiconformal mapping $\varphi: \Omega\to\Omega'$ such that
$$
\int\limits_\Omega |J(x, \varphi)|^{\beta}~dx < \infty \quad\text{for some}\quad \beta >1.
$$

Taking into account the quasiconformal inequality $|D\varphi(x)|^n\leq K |J(x,\varphi)|$ for almost all $x\in\Omega$ we obtain
\begin{multline*}
K_{p,q}^{\frac{pq}{p-q}}(\Omega)=\int\limits_\Omega \left(\frac{|D\varphi(x)|^p}{|J(x,\varphi)|}\right)^\frac{q}{p-q}~dx=
\int\limits_\Omega \left(\frac{|D\varphi(x)|^n}{|J(x,\varphi)|}|D\varphi(x)|^{p-n}\right)^\frac{q}{p-q}~dx\\
\leq
K^{\frac{q}{p-q}}\int\limits_\Omega \left(|D\varphi(x)|^{p-n}\right)^\frac{q}{p-q}~dx =
K^{\frac{q}{p-q}}\int\limits_\Omega |D\varphi(x)|^{\beta}~dx<\infty,
\end{multline*}
for $\beta=(p-n)q/(p-q)$. Hence by \cite{U93,VU02} we have a bounded composition operator
$$
\varphi^*: L_p^1(\Omega') \to L_q^1(\Omega)
$$
for any $p \in (n\,, +\infty)$ and $q=p\beta/(p+\beta -n)$.

Let us check that $q<p$. Because $p>n$ we have that $p+\beta -n> \beta >1$ and so $\beta /(p+\beta -n) <1$. Hence we obtain  $q<p$.

Assume that the composition operator
$$
\varphi^*: L_p^1(\Omega') \to L_q^1(\Omega),\,\,q<p,
$$
is bounded for any $p \in (n\,, +\infty)$ and $q=p\beta/(p+\beta -n)$.
Then, given the  Hadamard inequality:
$$
|J(x,\varphi)|\leq |D\varphi(x)|^n\,\,\text{for almost all}\,\, x\in\Omega,
$$
and by \cite{U93,VU02} we have
$$
\int\limits _{\Omega}|D\varphi(x)|^{\beta}~dx=\int\limits _{\Omega}|D\varphi(x)|^{\frac{(p-n)q}{p-q}}~dx\leq
\int\limits_\Omega \left(\frac{|D\varphi(x)|^{p}}{|J(x,\varphi)|}\right)^\frac{q}{p-q}~dx<+\infty.
$$
\end{proof}

\begin{remark}
In the case of bounded domains $\Omega, \Omega'\subset \mathbb R^n$,
$n\geq 2$, Theorem~\ref{th-conn} is correct for
any $p \in (n, +\infty)$ and any $q\leq p\beta/(p+\beta -n)$.

If $q> n-1$, then by the duality composition theorem \cite{U93}, the inverse mapping $\varphi^{-1}$ induces a bounded composition operator from
$L^1_{q'}(\Omega)$ to $L^1_{p'}(\Omega')$, where $p'={p}/{(p-(n-1))}$ and $q'={q}/{(q-(n-1))}$.
\end{remark}

\section{Sobolev-Poincar\'e inequalities}

\subsection{Two-weight Sobolev-Poincar\'e inequalities} Let $\Omega \subset \mathbb R^n$, $n \geq 2$,
be a domain and let $h : \Omega \to \mathbb R$ be a real valued locally integrable function such that $h(x)>0$ a.e. in $\Omega$. We consider the weighted Lebesgue space $L_p(\Omega,h)$, $1\leq p<\infty$, as the space
of measurable functions $f: \Omega \to \mathbb R$  with the finite norm
$$
\|f\,|\,L_{p}(\Omega,h)\|:= \left(\int\limits_\Omega|f(x)|^ph(x)~dx \right)^{\frac{1}{p}}< \infty.
$$
It is a Banach space for the norm $\|f\,|\,L_{p}(\Omega,h)\|$.

On the basis of Theorem~\ref{NormCO} we prove the existence of two-weight Sobolev-Poincar\'e inequalities in quasiconformal equivalent  bounded space domains.

Recall that a bounded domain $\Omega\subset\mathbb R^n$ is called a $(r,q)$-Sobolev-Poincar\'e domain, $1\leq r,q\leq \infty$, if for any function $f\in L^1_q(\Omega)$, the $(r,q)$-Sobolev-Poincar\'e inequality
$$
\inf\limits_{c\in\mathbb R}\|f-c\mid L_r(\Omega)\|\leq B_{r,q}(\Omega)\|\nabla f\mid L_q(\Omega)\|
$$
holds. Note that bounded Lipschitz domains $\Omega\subset\mathbb R^n$ are $(r,q)$-Sobolev-Poincar\'e domains, for $1\leq r\leq nq/(n-q)$, if $1\leq q<n$, and for any $r\geq 1$, if $q\geq n$ (see, for example, \cite{M}).

\begin{theorem}\label{Th3.1}
Let $\Omega, \Omega'\subset \mathbb R^n$, $n\geq 2$, be $K$-quasiconformal equivalent bounded domains and let $h(y)=|J(y,\varphi^{-1})|$ be the quasiconformal weight defined by a $K$-quasiconformal mapping $\varphi:\Omega\to\Omega'$. Suppose that $\Omega$ be a $(r,q)$-Sobolev-Poincar\'e domain, then for any function $f\in W_p^1(\Omega')$, $p>n$, the inequality
\[
\inf\limits_{c \in \mathbb R}\left(\int\limits_{\Omega'} |f(y)-c|^rh(y)~dy\right)^{\frac{1}{r}} \leq B_{r,p}(\Omega',h)
\left(\int\limits_{\Omega'} |\nabla f(y)|^p~dy\right)^{\frac{1}{p}}
\]
holds for any $1\leq r \leq nq/(n-q)$ with the constant
$$
B_{r,p}(\Omega,h) \leq \inf\limits_{q \in [1, n]} \left\{B_{r,q}(\Omega) |\Omega|^{\frac{n-q}{nq}}\right\} K^{\frac{1}{n}} |\Omega'|^{\frac{p-n}{np}}.
$$
\end{theorem}

Here $B_{r,q}(\Omega)$ is the best constant in the (unweighted) $(r,q)$-Sobolev-Poincar\'e inequality in the domain $\Omega$.

\begin{proof}
By the conditions of the theorem there exists a $K$-quasiconformal mapping $\varphi:\Omega\to\Omega'$. Denote by $h(y)=|J(y,\varphi^{-1})|$ the quasiconformal weight in $\Omega'$.

Let $f \in L_p^1(\Omega')$ be a smooth function. Then the composition $g=f \circ \varphi^{-1}$ is well defined almost everywhere in $\Omega$ and belongs to the Sobolev space $L_q^1(\Omega)$ \cite{VGR}. Hence, by the Sobolev embedding theorem $g=f \circ \varphi^{-1} \in W_q^1(\Omega)$ \cite{M} and the unweighted Poincar\'e-Sobolev inequality
\begin{equation}\label{IN3.1}
\inf_{c \in \mathbb R}||f \circ \varphi^{-1} -c \,|\, L_{r}(\Omega)|| \leq B_{r,q}(\Omega) ||\nabla (f \circ \varphi^{-1}) \,|\, L_{q}(\Omega)||
\end{equation}
holds for any $1\leq r \leq nq/(n-q)$.

Taking into account the change of variable formula for quasiconformal mappings \cite{VGR}, the Poincar\'e-Sobolev inequality \eqref{IN3.1} and Theorem \ref{NormCO}, we obtain for a smooth function $f \in W_p^1(\Omega')$
\begin{multline*}
\inf\limits_{c \in \mathbb R}\left(\int\limits_{\Omega'} |f(y)-c|^rh(y)dy\right)^{\frac{1}{r}} =
\inf\limits_{c \in \mathbb R}\left(\int\limits_{\Omega'} |f(y)-c|^r |J(y,\varphi^{-1})| dy\right)^{\frac{1}{r}} \\
{} = \inf\limits_{c \in \mathbb R}\left(\int\limits_{\Omega} |g(x)-c|^rdx\right)^{\frac{1}{r}} \leq B_{r,q}(\Omega)
\left(\int\limits_{\Omega} |\nabla g(x)|^q dx\right)^{\frac{1}{q}} \\
{} \leq B_{r,q}(\Omega) K^{\frac{1}{n}} |\Omega|^{\frac{n-q}{nq}} |\Omega'|^{\frac{p-n}{np}}
\left(\int\limits_{\Omega'} |\nabla f(y)|^p dy\right)^{\frac{1}{p}}.
\end{multline*}

Approximating an arbitrary function $f \in W^{1}_{p}(\Omega')$ by smooth functions we have
$$
\inf\limits_{c \in \mathbb R}\left(\int\limits_{\Omega'} |f(y)-c|^rh(y)dy\right)^{\frac{1}{r}} \leq
B_{r,p}(\Omega',h) \left(\int\limits_{\Omega'} |\nabla f(y)|^p dy\right)^{\frac{1}{p}}
$$
with the constant
$$
B_{r,p}(\Omega,h) \leq \inf\limits_{q \in [1, n]} \left\{B_{r,q}(\Omega) |\Omega|^{\frac{n-q}{nq}}\right\} K^{\frac{1}{n}} |\Omega'|^{\frac{p-n}{np}}.
$$
\end{proof}

The property of the $K$-quasiconformal $\beta$-regularity implies the integrability of a
Jacobian of quasiconformal mappings and therefore for any $K$-quasiconformal $\beta$-regular domain we have the embedding of weighted Lebesgue spaces $L_r(\Omega,h)$ into non-weight Lebesgue
spaces $L_s(\Omega)$ for $s={(\beta -1)r}/{\beta}$ \cite{GPU19}:

\begin{lemma} \label{L3.2}
Let $\Omega'$ be a $K$-quasiconformal $\beta$-regular domain about a domain $\Omega$. Then for any function
$f \in L_r(\Omega,h)$, $\beta / (\beta - 1) \leq r < \infty$, the inequality
$$
||f\,|\,L_s(\Omega')|| \leq \left(\int\limits_\Omega \big|J(x,\varphi)\big|^{\beta}~dx \right)^{{\frac{1}{\beta}} \cdot \frac{1}{s}} ||f\,|\,L_r(\Omega',h)||
$$
holds for $s={(\beta -1)r}/{\beta}$.
\end{lemma}

According to Theorem \ref{Th3.1} and Lemma \ref{L3.2} we obtain an upper estimate of the Poincar\'e constant in quasiconformal regular domains.

\begin{theorem}\label{Th3.3}
Let $\Omega'$ be a $K$-quasiconformal $\beta$-regular domain about a $(r,q)$-Sobolev-Poincar\'e domain $\Omega$. Then for any function $f \in W_p^1(\Omega')$, $p>n$, the Poincar\'e-Sobolev inequality
$$
\inf\limits_{c \in \mathbb R}\left(\int\limits_{\Omega'} |f(y)-c|^s h(y)dy\right)^{\frac{1}{r}} \leq
B_{s,p}(\Omega') \left(\int\limits_{\Omega'} |\nabla f(y)|^p dy\right)^{\frac{1}{p}}
$$
holds for any $1\leq s\leq nq/(n-q)\cdot (\beta -1)/\beta$ with the constant
\begin{multline*}
B_{s,p}(\Omega') \leq \left(\int\limits_\Omega \big|J(x,\varphi)\big|^{\beta}~dx \right)^{{\frac{1}{\beta}} \cdot \frac{1}{s}} B_{r,p}(\Omega,h) \\
\leq \inf\limits_{q \in (q^{\ast}, n]} \left\{B_{r,q}(\Omega) |\Omega|^{\frac{n-q}{nq}}\right\} K^{\frac{1}{n}} |\Omega'|^{\frac{p-n}{np}} \cdot ||J_{\varphi}\,|\,L_{\beta}(\Omega)||^{\frac{1}{s}},
\end{multline*}
where $q^{\ast}=\beta ns/(\beta s +\beta (n-1))$
\end{theorem}

\begin{proof}
Let $f \in W^1_p(\Omega')$, $p>n$. Then by Theorem~\ref{Th3.1} and Lemma~\ref{L3.2} we obtain
\begin{multline*}
\inf\limits_{c \in \mathbb R}\left(\int\limits_{\Omega'} |f(y)-c|^s~dy\right)^{\frac{1}{s}} \\
{} \leq  \left(\int\limits_\Omega \big|J(x,\varphi)\big|^{\beta}~dx\right)^{{\frac{1}{\beta}} \cdot \frac{1}{s}}
\inf\limits_{c \in \mathbb R}\left(\int\limits_{\Omega'} |f(y)-c|^rh(y)~dy\right)^{\frac{1}{r}} \\
{} \leq B_{r,p}(\Omega', h) \left(\int\limits_\Omega \big|J(x,\varphi)\big|^{\beta}~dx\right)^{{\frac{1}{\beta}} \cdot \frac{1}{s}}
\left(\int\limits_{\Omega'} |\nabla f(y)|^p~dy\right)^{\frac{1}{p}}
\end{multline*}
for $1\leq s\leq nq/(n-q)\cdot (\beta -1)/\beta$.

Since by Lemma~\ref{L3.2} $s=\frac{\beta -1}{\beta}r$ and by Theorem~\ref{Th3.1} $r \geq 1$, then $s \geq 1$ and the theorem proved.
\end{proof}

By the generalized version of the Rellich-Kondrachov compactness theorem (see, for example, \cite{M}) and the $(r,p)$--Sobolev-Poincar\'e inequality for $r>p$, it follows that
the embedding operator
$$
i: W^1_p(\Omega) \hookrightarrow L_p(\Omega)
$$
is compact in $K$-quasiconformal $\beta$-regular domains $\Omega \subset \mathbb R^n$, $n \geq 2$.
Note that  sufficient conditions for validity of  the Rellich-Kondrachov theorem  in non-smooth domains have been given by using a general quasihyperbolic boundary condition  in \cite{EHS}.
In particular, for domains $\Omega \subset \mathbb R^n$, $n \geq 2$, satisfying a quasihyperbolic boundary condition it is proved that there exists $p_0=p_0 (\Omega)<n$ such that
the embedding operator 
$$
i: W^1_p(\Omega) \hookrightarrow L_p(\Omega)
$$
is compact for every $p>p_0$.

So, by the Min-Max Principle the first non-trivial Neumann eigenvalue $\mu_p(\Omega)$ can be characterized \cite{ENT} as
\begin{equation*}
\mu_p(\Omega)\\=\min \left\{\frac{\int\limits_\Omega |\nabla u(x)|^{p}~dx}{\int\limits_\Omega |u(x)|^{p}~dx} : u \in W_p^1(\Omega) \setminus \{0\},
\int\limits_\Omega |u|^{p-2}u~dx=0 \right\}.
\end{equation*}


In the case $s=p$ Theorem~\ref{Th3.3} and the Min-Max Principle implies the lower estimates of the first non-trivial eigenvalue of the degenerate $p$-Laplace Neumann
operator, $p>n$, in $K$-quasiconformal $\beta$-regular domains $\Omega' \subset \mathbb R^n$, $n \geq 2$.

\vskip 0.2cm
\noindent
\begin{theorem}
\label{thm:est}
Let $\Omega'$ be a $K$-quasiconformal $\beta$-regular domain about a $(r,q)$-Sobolev-Poincar\'e domain $\Omega$, $r=p\beta/(\beta-1)$, $p>n$.
Then the following inequality holds
$$
\frac{1}{\mu_p(\Omega')}
 \leq  \inf\limits_{q \in (q^{\ast}, n]} \left\{B_{r,q}^p(\Omega) |\Omega|^{\frac{p(n-q)}{nq}}\right\} K^{\frac{p}{n}} |\Omega'|^{\frac{p-n}{n}} \cdot ||J_{\varphi}\,|\,L_{\beta}(\Omega)||,
$$
where $q^{\ast}=\beta np/(\beta p + n(\beta-1))$.
\end{theorem}

In the case of $K$-quasiconformal $\infty$-regular domains, in the similar way we obtain the following assertion:

\vskip 0.2cm
\noindent
\begin{theorem}
\label{thm:est2}
Let $\Omega'$ be a $K$-quasiconformal $\infty$-regular domain about a $(p,q)$-Sobolev-Poincar\'e domain $\Omega$.
Then for any $p>n$ the following inequality holds
$$
\frac{1}{\mu_p(\Omega')}
 \leq  \inf\limits_{q \in (q^{\ast}, n]} \left\{B_{p,q}^p(\Omega) |\Omega|^{\frac{p(n-q)}{nq}}\right\} K^{\frac{p}{n}} |\Omega'|^{\frac{p-n}{n}} \cdot ||J_{\varphi}\,|\,L_{\infty}(\Omega)||,
$$
where $q^{\ast}=np/(p + n)$.
\end{theorem}

Note that any convex domain $\Omega \subset \mathbb R^n$ is  the Sobolev-Poincar\'e domain and the constant
$B_{r,q}(\Omega)$ can be estimated as \cite{GU16}
\[
B_{r,q}(\Omega)\leq \frac{d^n_{\Omega}}{n|\Omega|}\left(\frac{1-\frac{1}{q}+\frac{1}{r}}{\frac{1}{n}-\frac{1}{q}+\frac{1}{r}}\right)^{1-\frac{1}{q}+\frac{1}{r}} \omega_n^{1-\frac{1}{n}}
|\Omega|^{\frac{1}{n}-\frac{1}{q}+\frac{1}{r}},
\]
where $\omega_n=\frac{2 \pi^{n/2}}{n \Gamma (n/2)}$ is the volume of the unit ball in $\mathbb R^n$ and $d_{\Omega}$ is the diameter of $\Omega$.

\vskip 0.2cm

As examples, we consider non-convex star-shaped domains which are $K$-quasiconformal $\infty$-regular domains.

\begin{example}
The homeomorphism
\[
\varphi(x)=|x|^{a}x, \quad a>0,
\]
is $(a+1)$-quasiconformal and maps the $n$-dimensional cube
\[
Q:=\{x_k \in \mathbb R^n : |x_k|<\sqrt{2}/2\}
\]
onto non-convex star-shaped domains $\Omega_a$.

Now we estimate the following quantities. A straightforward calculation yields
\[
||J_{\varphi}\,|\,L_{\infty}(Q)||= \esssup\limits_{x \in Q} \left[(a+1)|x|^{na}\right] \leq a+1
\]
and
\[
d_Q=2, \quad |Q|=2^{n/2}, \quad |\Omega_a| \leq \omega_n.
\]

Then by Theorem~\ref{thm:est2} we have
\begin{multline*}
\frac{1}{\mu_p(\Omega_a)}
 \leq  \inf\limits_{q \in (q^{\ast}, n]} \left\{B_{p,q}^p(Q) |Q|^{\frac{p(n-q)}{nq}}\right\} K^{\frac{p}{n}} |\Omega_a|^{\frac{p-n}{n}} \cdot ||J_{\varphi}\,|\,L_{\infty}(Q)|| \\
 \leq \inf\limits_{q \in (q^{\ast}, n]}
 \left(\frac{1-\frac{1}{q}+\frac{1}{p}}{\frac{1}{n}-\frac{1}{q}+\frac{1}{p}}\right)^{p+1-\frac{p}{q}}
 \frac{2^{\frac{n(p+1)}{2}}}{n^p} (a+1)^{\frac{p}{n}} \omega_n^{p+1-\frac{p}{n}},
\end{multline*}
where $q^{\ast}=np/(p + n)$.
\end{example}

In the case of quasiconformal mappings $\varphi:\mathbb B\to\Omega$ Theorem~\ref{thm:est} can be reformulated as

\begin{theorem}
\label{thm:estball}
Let $\Omega$ be a $K$-quasiconformal $\beta$-regular domain about the unit ball $\mathbb B$, $r=p\beta/(\beta-1)$, $p>n$.
Then the following inequality holds
\begin{equation*}
\frac{1}{\mu_p(\Omega)}
\leq \\
\inf\limits_{q \in (q^{\ast}, n]} \left\{\frac{2^{np}}{n^p}\left(\frac{1-\frac{1}{q}+\frac{1}{r}}{\frac{1}{n}-\frac{1}{q}+\frac{1}{r}}\right)^{p-\frac{p}{q}+\frac{p}{r}} \omega_n^{\frac{p}{r}-\frac{p}{n}}\right\} K^{\frac{p}{n}} |\Omega|^{\frac{p-n}{n}} \cdot ||J_{\varphi}\,|\,L_{\beta}(\mathbb B)||,
\end{equation*}
where $q^{\ast}=\beta np/(\beta p + n(\beta-1))$.
\end{theorem}

\section{The weak reverse H\"older inequality}

In this section we obtain estimates of constants in the reverse H\"older inequality. In this section we suppose that $n\geq 3$. The case $n=2$ was considered in \cite{GPU18_2,GPU18}. We start from the following version of the weak reverse  H\"older inequality \cite{BI}.

\begin{proposition}\label{BojarskiIwaniecLemma}
\cite{BI}
Suppose that
$\varphi :\Omega\to \mathbb R^n$ is a $K$-quasiconformal mapping and $0<\sigma <1$ is given. Then, for any ball $B$ in $\Omega$ there exists a constant $C(n)$, depending only on $n$  such that
\begin{equation}\label{BIEquation}
\biggl(\frac{1}{\vert \sigma B\vert }
\int\limits_{\sigma B}\vert D\varphi (x)\vert^n\,dx\biggr)^{\frac{1}{n}}\le \frac{KC(n)}{\sigma (1-\sigma )}
\biggl(\frac{1}{\vert B\vert}
\int\limits_{B}\vert D\varphi (x)\,dx\vert ^{\frac{n}{2}}\biggr)^{\frac{2}{n}}\,.
\end{equation}
\end{proposition}
In inequality \eqref{BIEquation}
\begin{equation}\label{ConstantLemmaBI}
C(n)=2^{2n+{3}/{2}+{1}/{n}}
\bigl(
{n}/{(n-2)}\bigr)^{{1}/{n}}5 \omega_{n}\,,\quad  n>2\,.
\end{equation}

We need
the following special case of
\cite{IN} when
balls are used instead of cubes. This result comes  from an iteration process where  the
exponent  $n/2$ on the right hand side  of \eqref{BIEquation} can be reduced to $1$.

\begin{proposition}\label{IwaniecNolderBalls}
Let $\Omega$  be a bounded domain in $\mathbb R ^n$, $n\geq2$,
and let $f\in L^n_{\loc}(\Omega)$.
Suppose that there exists a constant $C_0$ such that
for each ball $B$ with $2B\subset\Omega$
\begin{equation*}
\biggl(\frac{1}{\vert B\vert }
\int\limits_{B}\vert f(x)\vert^n\,dx\biggr)^{\frac{1}{n}}\le C_0
\biggl(\frac{1}{\vert 2B\vert}
\int\limits_{ 2B}\vert f(x)\vert ^{\frac{n}{2}}\,dx\biggr)^{\frac{2}{n}}\,.
\end{equation*}
Then for each ball $B$ with $2B\subset \Omega$
\begin{equation}\label{reduced}
\biggl(\frac{1}{\vert B\vert }
\int\limits_{B}\vert f(x)\vert^n\,dx\biggr)^{\frac{1}{n}}\le C_1
\frac{1}{\vert 2B\vert}
\int\limits_{2B}\vert f(x)\vert \,dx\,.
\end{equation}
\end{proposition}

We apply the following result from \cite{I} in the special case when $g, h\in L^n(\Omega )$.



\begin{proposition}\label{IwaniecBalls}
Let $\Omega$ be a ball in $\mathbb R^n$ and let $g, h\in L^n(\Omega)$, $n\geq 2$,
be nonnegative functions satisfying:
\begin{equation*}
\biggl(\frac{1}{\vert B\vert }
\int\limits_{B}g(x)^n\,dx\biggr)^{\frac{1}{n}}\le
C_1
\frac{1}{\vert 2B\vert }
\int\limits_{2B}g(x)\,dx
+
\biggl(\frac{1}{\vert 2B\vert}
\int\limits_{ 2B} h(x)^n\,dx\biggr)^{\frac{1}{n}}
\end{equation*}
for all balls $B$ with $2B\subset\Omega$. Then for each $0<\sigma <1$ and $n<p<n+{(n-1)}/{10^{2n}4^n C_1^n}$
we have
\begin{equation}\label{TwoTermsIneq}
\biggl(\frac{1}{\vert \sigma\Omega\vert }
\int\limits_{\sigma\Omega}g(x)^p\,dx\biggr)^{\frac{1}{p}}
\le
\frac{100^n}{\sigma^{{n}/{p}}(1-\sigma )}
\biggl[
\biggl(
\frac{1}{\vert \Omega \vert }
\int\limits_{\Omega}g(x)^n\,dx\biggr)^{\frac{1}{n}}
+
\biggl(\frac{1}{\vert \Omega\vert}
\int\limits_{ \Omega} h(x)^p\,dx\biggr)^{\frac{1}{p}}
\biggr]\,.\notag
\end{equation}

\end{proposition}

Now we are able to obtain the weak reverse H\"older inequality with bounds for the constants.

\begin{theorem}\label{WeakReverseResult}
Let   $\Omega$ be a bounded domain in  $\mathbb R^n$, $n\geq 2$ such that
$B(0,2)\subset\Omega$.
Suppose that $\varphi :\Omega\to\mathbb R^n$ is a $K$-quasiconformal mapping
such that $C_1$ is the constant from \eqref{reduced} for $D\varphi$.
Then,
for any $p>n$ satisfying $n<p<n+{(n-1)}/{10^{2n}4^n C_1^n}$
we have
\begin{equation}\label{weakreverse}
\biggl(\frac{1}{\vert B(0,1)\vert }
\int\limits_{ B(0,1)}\vert D \varphi (x)\vert^p\,dx\biggr)^{\frac{1}{p}}\le C(n,p)
\biggl(\frac{1}{\vert B(0,2)\vert}
\int\limits_{ B(0,2)}\vert D\varphi (x)\vert^n\,dx\biggr)^{\frac{1}{n}}\,.
\end{equation}

\end{theorem}

\begin{proof}
By Proposition \ref{BojarskiIwaniecLemma}  the assumptions in Proposition \ref{IwaniecNolderBalls} are valid. Hence,
for all the balls $B$ with $2B\subset\Omega$ the inequality
\begin{equation*}
\biggl(\frac{1}{\vert B\vert }
\int\limits_{B}\vert D\varphi  (x)\vert^n\,dx\biggr)^{\frac{1}{n}}\le c_0(n)(4KC(n))^{\frac{2(n-1)}{n}}
\frac{1}{\vert 2B\vert}
\int\limits_{2 B}\vert D\varphi (x)\vert \,dx
\end{equation*}
holds with $C(n)$ from \eqref{ConstantLemmaBI}.
We write $C_1=c_{0}(n)(4KC(n))^{\frac{2(n-1)}{n}}$.
Thus the assumptions of Proposition \ref{IwaniecBalls} are valid
when $h=0$ and  $\sigma =1/2$.
Hence
by Proposition \ref{IwaniecBalls}
for any $p>n$ satisfying $n<p<n+{(n-1)}/{10^{2n}4^n C_1^n}$
we have
\begin{equation}\label{weakreverse1}
\biggl(\frac{1}{\vert B(0,1)\vert }
\int\limits_{ B(0,1)}\vert D \varphi (x)\vert^p\,dx\biggr)^{\frac{1}{p}}\le C(n,p)
\biggl(\frac{1}{\vert B(0,2)\vert}
\int\limits_{ B(0,2)}\vert D\varphi (x)\vert^n\,dx\biggr)^{\frac{1}{n}}
\end{equation}
where $C(n,p)=2^{1+{n}/{p}} (100)^n$.
\end{proof}

Now we need the doubling constant for the Jacobian of a given quasiconformal mapping. For this aim we use the moduli (capacity) estimates.
Let $\Gamma$ be a family of curves in $\mathbb R^n$. Denote by $adm(\Gamma)$ the set of Borel functions (admissible functions)
$\rho: \mathbb R^n\to[0,\infty]$ such that the inequality
$$
\int\limits_{\gamma}\rho~ds\geqslant 1
$$
holds for locally rectifiable curves $\gamma\in\Gamma$.

Let $\Gamma$ be a family of curves in $\overline{\mathbb R^n}$, where $\overline{\mathbb R^n}$ is a one point compactification of the Euclidean space $\mathbb R^n$. The quantity
$$
M(\Gamma)=\inf\int\limits_{\mathbb R^n}\rho^{n}~dx
$$
is called the module of the family of curves $\Gamma$ (see, for example \cite{Vuo}). The infimum is taken over all admissible functions
$\rho\in adm(\Gamma)$.

Let $\Omega$ be a bounded domain in $\mathbb R^n$ and $F_0, F_1$ be disjoint non-empty compact sets in the
closure of $\Omega$, then $M(\Gamma(F_0,F_1;\Omega))$ stand for the
module of a family of curves which connect $F_0$ and $F_1$ in $\Omega$

\begin{theorem}\label{doubling}
Suppose that  $\varphi :\mathbb R^n\to\mathbb R^n$  is a $K$-quasiconformal mapping. Then
for any ball $B$,
\begin{equation*}
\int\limits_{2B}|J(x,\varphi)|\,dx\le
\exp \biggl\{K^{1/(n-1)}2(\log (\sqrt{3}+\sqrt{2})+n-1)\biggr\}
\int\limits_{B}|J(x,\varphi)|\,dx\,.
\end{equation*}
\end{theorem}

\begin{proof}
When $n$ and $K$ are given,
our goal is to find the constant $C$ such that
\begin{equation}\label{Find_C}
\frac{\vert \varphi(B(0,2))\vert}{\vert \varphi(B(0,1))\vert}\le C\,.
\end{equation}
That is, to find $C$ such that
\begin{equation*}
\int\limits_{B(0,2)}|J(x,\varphi)|\,dx\le C
\int\limits_{B(0,1)}|J(x,\varphi)|\,dx\,.
\end{equation*}
Let us write
$R=\max\vert \varphi(0)-y\vert \mbox{ when } y\in\partial \varphi(B(0,2))$
and
$r=\dist (\varphi(0),\partial \varphi(B(0,1))$.

Denote by $L_r$ a line segment of length $r$ joining $\varphi(0)$ to $\partial \varphi(B(0,1))$ and
$L_R$ a continuum joining a point in $\partial B(\varphi(0),R)\cap\partial \varphi(B(0,2))$ to $\infty$ in
$\mathbb R^n\backslash B(\varphi(0),R)$.
Then by the quasi-invariance of the modulus
\begin{multline*}
0< C_1\le \mod (\varphi^{-1}(L_r),\varphi^{-1}(L_R), B(0,10))\\
\le K\mod (L_r,L_R,\mathbb R^n)\le K{\omega_{n-1}}{\biggl(\log\frac{R}{r}\biggr)^{1-n}}\,,
\end{multline*}
where $\omega_{n-1}=\frac{2 \pi^{n/2}}{\Gamma (n/2)}$ is the hypervolume of the $(n-1)$-dimensional unit sphere.

We write
\begin{equation}\label{Teichmuller}
R_T(n,t)=R([-1,0], [t,\infty );\mathbb R^n )
=\mathbb R^n\backslash\{[-e_1,0]\cup [te_1,\infty )\}\,, t>0
\end{equation}
for the $n$-dimensional Teichm\"uller ring corresponding $t$ and also
$\tau_n(t)=\mod (R_T(n,t) )$.
Let us recall that
\begin{equation*}
\mod (\Delta (E,F))\geq \tau _n\bigg(\frac{\vert a-c\vert}{\vert a-b\vert}\biggr)\,.
\end{equation*}
This equality holds for $E=[-e_1,0]$, $a=0, b=-e_1$ and
$F=[te_1,\infty )$, $c=te_1$\,, $d=\infty$
(see, for example \cite{Ra}).

Let
$R(E,F))$ be a ring with $a,b\in E$ and $c,\infty\in F$.
Then  for the Teichm\"uller ring $R_T(n,t)$ we have the estimate
$$
\mod (R(E,F))\geq \mod R_T\bigg(n,\frac{\vert a-c\vert}{\vert a-b\vert}\biggr),
$$
where 
$$
\mod (R_T(n,t))=\tau_n(t),\,\,\,
\omega_{n-1}(\log (\lambda_n^2t))^{1-n}\le\tau_n(t-1)\le\omega_{n-1}(\log (t))^{1-n},
$$
see, for example \cite{GMP}.

By \cite{Ra}, based on \cite{Vuo},
 \begin{equation*}
\tau_n(t)\geq 2^{1-n}\omega_{n-1}\biggl(\log\frac{\lambda_n}{2}(\sqrt{1+t}+\sqrt{t})\biggr)^{1-n}\,,\,\, t>0\,,
\end{equation*}
where $\lambda_n$ is the Gr\"otzsch ring constant depending only on $n$.
It is know that $\lambda_2=4$, and for $n\geq 3$ it is known only that
$2^{0.76(n-1)}\le\lambda_n\le 2e^{n-1}$.
This gives a lower bound
\begin{equation*}
\tau_n(t=2)\geq 2^{1-n}\omega_{n-1}\biggl(\log\frac{\lambda_n}{2}(\sqrt{3}+\sqrt{2})\biggr)^{1-n}\,,\,\,\text{where}\,\,4\le \lambda_n\le 2e^{n-1}\,.
\end{equation*}

Hence we obtain
$$
C_1=2^{1-n}\omega_{n-1}\biggl(\log\frac{\lambda_n}{2}(\sqrt{3}+\sqrt{2})\biggr)^{1-n}
\le K\omega_{n-1}\biggl(\log\frac{R}{r}\biggr)^{1-n}\,.
$$
Thus,
\begin{multline*}
\biggl(\log\frac{R}{r}\biggr)^{n-1}
\le  \frac{K\omega_{n-1}}{2^{1-n}\omega_{n-1}\biggl(\log\frac{\lambda_n}{2}(\sqrt{3}+\sqrt{2})\biggr)^{1-n}}\\
\le K \biggl(2\log (\frac{\lambda_n}{2}(\sqrt{3}+\sqrt{2}))\biggr)^{n-1}\,,
\end{multline*}
and
\begin{equation*}
\log\frac{R}{r}\le K^{1/(n-1)}2\log(\frac{\lambda_n}{2}(\sqrt{3}+\sqrt{2}))\,.
\end{equation*}
Hence
\begin{equation*}
\frac{R}{r}\le \exp \biggl\{K^{1/(n-1)}2\log(\frac{\lambda_n}{2}(\sqrt{3}+\sqrt{2}))\biggr\}\,,
\end{equation*}
and since $\lambda_n < 2e^{n-1}$, we obtain
\begin{multline*}
\frac{R}{r}\le \exp \biggl\{K^{1/(n-1)}2\log(\sqrt{3}+\sqrt{2})e^{n-1}\biggr\}
\\
= \exp \biggl\{K^{1/(n-1)}2(\log(\sqrt{3}+\sqrt{2})+n-1)\biggr\}\,.
\end{multline*}

\end{proof}

Using the previous results concerning the weak
reverse H\"older inequality and the measure doubling condition we obtain
an estimate of the constant in the reverse H\"older inequality for Jacobians of
quasiconformal mappings.

\begin{theorem}\label{th-rhin}
Let $\Omega$ be a bounded domain in $\mathbb R^n$ such that $B(0,2)\subset\Omega$.
Suppose that $\varphi : \Omega\to\Omega'$ is a $K$-quasiconformal mapping.
Then there exists a constant $C(n,\alpha , K)$  such that
\begin{equation*}
\biggl(\int\limits_{B(0,1)}\vert J(x,\varphi)\vert ^{\frac{\alpha}{ n}}\,dx\biggr)^{\frac{n}{\alpha}}
\le C(n,\alpha , K)
\int\limits_{B(0,1)} |J(x,\varphi)|\,dx
\end{equation*}
for all $\alpha\in \left(n, n+{(n-1)}/{10^{2n}4^nC_1^n}\right)$
where
\[
C_1=2^{2n+3/2+1/n}\bigl(
{n}/{(n-2)}\bigr)^{1/n}5 \omega_{n} (4K)^{2(n-1)/n}\,,\quad  n>2,
\]
and
\[
C(n,\alpha , K)=2^{1-n+n/\alpha}(100)^n K \omega_{n}^{n/{\alpha -1}}
\exp \biggl\{K^{1/(n-1)}2(\log(\sqrt{3}+\sqrt{2})+n-1)\biggr\}\,.
\]
\end{theorem}

\begin{proof}
By Proposition \ref{BojarskiIwaniecLemma} inequality \eqref{BIEquation} is valid
where  equation \eqref{ConstantLemmaBI} gives a constant.
Hence, by Proposition \ref{IwaniecNolderBalls}  inequality \eqref{reduced}
is valid  for $D\varphi$ with the constant $C_1$.
Let us choose
\begin{equation*}
\alpha \in \biggl(n, n+\frac{n-1}{10^{2n}4^nC_1^n}\biggr).
\end{equation*}
Now for these values of $\alpha$ by Theorem \ref{WeakReverseResult} we obtain
\begin{align*}
\biggl(
\int\limits_{B(0,1)}\vert J(x,\varphi)\vert ^{\frac{\alpha}{n}}\,dx
\biggr)^{\frac{n}{\alpha}}
&\le
\vert B(0,1)\vert ^{n/{\alpha}}
\biggl(
\frac{1}{\vert B(0,1)\vert}
\int\limits_{B(0,1)}\vert D\varphi (x)\vert ^{\alpha}\,dx
\biggr)^{\frac{n}{\alpha}}\\
&\le
2^{1+n/\alpha}(100)^n
\frac{\vert B(0,1)\vert^{n/{\alpha}}}{\vert B(0,2)\vert }
\int\limits_{B(0,2)}\vert D\varphi (x)\vert ^{n}\,dx\,.
\end{align*}
By the definition of the quasiconformal mappings and the doubling condition (Theorem \ref{doubling})
\begin{multline*}
\biggl(
\int\limits_{B(0,1)}\vert J(x,\varphi)\vert ^{\frac{\alpha}{n}}\,dx
\biggr)^{\frac{n}{\alpha}}
\le
2^{1-n+n/\alpha}(100)^n K
\omega_{n}^{n/{\alpha}-1}
\int\limits_{B(0,2)}\vert J(x,\varphi)\vert \,dx\\
\le
C(\alpha , n, K)
\int\limits_{B(0,1)}\vert J(x,\varphi)\vert \,dx\,.
\end{multline*}

\end{proof}

Given this theorem, we can precise Theorem~\ref{thm:estball}, putting $\alpha/n=\beta$.

\begin{theorem}
\label{thm:main}
Let $\Omega$ be a $K$-quasiconformal $\beta$-regular domain about the unit ball $\mathbb B$, $r=p\beta/(\beta-1)$, $p>n$.
Then
\[
\frac{1}{\mu_p(\Omega)}\leq \\ \inf\limits_{q \in (q^{\ast}, n]} \left\{\frac{2^{np}}{n^p}\left(\frac{1-\frac{1}{q}+\frac{1}{r}}{\frac{1}{n}-\frac{1}{q}+\frac{1}{r}}\right)^{p-\frac{p}{q}+\frac{p}{r}} \omega_n^{\frac{p}{r}-\frac{p}{n}}\right\} K^{\frac{p}{n}} C(n, \beta, K) \cdot|\Omega|^{p/n}
\]
for all $\beta \in \left(1, 1+{(n-1)}/{n10^{2n}4^nC_1^n}\right)$,
where $q^{\ast}=\beta np/(\beta p +n(\beta-1))$ and
\[
C_1=2^{2n+3/2+1/n}\bigl(
{n}/{(n-2)}\bigr)^{1/n}5 \omega_{n} (4K)^{2(n-1)/n}\,,\quad  n>2,
\]
\[
C(n,\beta , K)=2^{1-n+1/\beta}(100)^n K \omega_{n}^{1/\beta -1}
\exp \biggl\{K^{1/(n-1)}2(\log(\sqrt{3}+\sqrt{2})+n-1)\biggr\}\,.
\]
\end{theorem}

\begin{proof}
According to Theorem~\ref{thm:estball} we have
\begin{multline}\label{Ineq-1}
\frac{1}{\mu_p(\Omega)}
\leq \\
\inf\limits_{q \in (q^{\ast}, n]} \left\{\frac{2^{np}}{n^p}\left(\frac{1-\frac{1}{q}+\frac{1}{r}}{\frac{1}{n}-\frac{1}{q}+\frac{1}{r}}\right)^{p-\frac{p}{q}+\frac{p}{r}} \omega_n^{\frac{p}{r}-\frac{p}{n}}\right\} K^{\frac{p}{n}} |\Omega|^{\frac{p-n}{n}} \cdot ||J_{\varphi}\,|\,L_{\beta}(\mathbb B)||,
\end{multline}
where $q^{\ast}=\beta np/(\beta p +n(\beta-1))$.
Now we estimate the integral from the right side of last inequality.
Given Theorem~\ref{th-rhin} for
$\beta \in \left(1, 1+{(n-1)}/{n10^{2n}4^nC_1^n}\right)$
we get
\begin{equation}\label{Ineq-2}
||J_{\varphi}\,|\,L_{\beta}(\mathbb B)||= \left(\int\limits_{\mathbb B} |J(x,\varphi)|^{\beta}\,dx \right)^{\frac{1}{\beta}}
\leq C(n, \beta ,K) |\Omega|.
\end{equation}
Combining inequalities~\eqref{Ineq-1} and~\eqref{Ineq-2} we obtain the required inequality.

\end{proof}

This theorem immediately follows

\begin{theorem}
\label{main}
Let a domain $\Omega\subset\mathbb R^n$, $n\geq 3$, satisfying the $\gamma$-quasihyperbolic boundary condition, be a $K$-quasi-ball. Then
$$
\mu_p(\Omega)\geq \frac{M_p(K,\gamma)}{R_{*}^p},
$$
where $R_{\ast}$ is a radius of a ball $\Omega^{\ast}$ of the same measure as $\Omega$ and $M_p(K,\gamma)$ depends only on $p,\gamma$ and the quasiconformity coefficient $K$ of $\Omega$.
\end{theorem}

\begin{proof}
Because the domain $\Omega$ satisfies the $\gamma$-quasihyperbolic boundary condition, it is a $K$-quasiconformal $\beta$-regular domain with some $\beta=\beta(\gamma)$. Then by the previous theorem
$$
\mu_p(\Omega)\geq \frac{M_p(K,\gamma)}{R_{*}^p},
$$
where
$$
M_p(K,\gamma)=\frac{2^{np}K^{\frac{p}{n}}}{n^p}\left(\frac{1-\frac{1}{q}+\frac{1}{r}}{\frac{1}{n}-\frac{1}{q}+\frac{1}{r}}\right)^{p-\frac{p}{q}+\frac{p}{r}} (\omega_n)^{\frac{p}{r}} C(n, \beta(\gamma), K),
$$
for some $q\in (q^{\ast},n]$.

\end{proof}

\vskip 0.3cm

\textbf{Acknowledgements.} The first author was supported by the United States-Israel Binational Science Foundation (BSF Grant No. 2014055). 
The second author,  whose visit to the Ben-Gurion University of Negev  was supported by a grant from the Finnish Academy of Science and Letters, Vilho, Yrj\"o and Kalle V\"ais\"al\"a Foundation, is grateful for the hospitality given by the Department of Mathematics of the Ben-Gurion University of the Negev. The third author
was supported by the Ministry of Science and Higher Education of Russia (agreement No. 075-02-2020-1479/1).

\vskip 0.3cm

\vskip 0.3cm

Department of Mathematics, Ben-Gurion University of the Negev, P.O.Box 653, Beer Sheva, 8410501, Israel

\emph{E-mail address:} \email{vladimir@math.bgu.ac.il} \\

Department of Mathematics and Statistics, Gustaf H\"allstr\"omin katu 2 b, FI-00014 University of Helsinki, Finland

\emph{E-mail address:} \email{ritva.hurri-syrjanen@helsinki.fi} \\

 Division for Mathematics and Computer Sciences, Tomsk Polytechnic University, 634050 Tomsk, Lenin Ave. 30, Russia;
 Regional Scientific and Educational Mathematical Center, Tomsk State University, 634050 Tomsk, Lenin Ave. 36, Russia
							
\emph{E-mail address:} \email{vpchelintsev@vtomske.ru}   \\
			
Department of Mathematics, Ben-Gurion University of the Negev, P.O.Box 653, Beer Sheva, 8410501, Israel
							
\emph{E-mail address:} \email{ukhlov@math.bgu.ac.il

\end{document}